\documentclass[10pt]{article} 

\usepackage{fullpage}
\usepackage{amsthm}
\usepackage{amssymb}
\usepackage{amsmath}
\usepackage[mathscr]{eucal}
\usepackage{graphicx}
\usepackage{centernot}
\usepackage{textcomp}
\usepackage{afterpage}
\usepackage{esint}
\usepackage{xspace}
\usepackage{siunitx}
\usepackage{tikz}
\usepackage{hyperref}
\usepackage{bm}
\usepackage{wasysym}
\usepackage{accents}
\usepackage{xcolor}
\usepackage{mathrsfs}

\usepackage[T1]{fontenc}
\usepackage[utf8]{inputenc}
\usepackage{authblk}

\newcommand{\abs}[1]{\left| #1 \right|}

\newcommand{\C}{\mathbb{C}}

\newcommand{\N}{\mathbb{N}}

\newcommand{\norm}[1]{\|{#1}\|}

\newcommand{\p}{\partial}

\newcommand{\R}{\mathbb{R}}

\newcommand{\set}[2]{\left\{{#1}:{#2}\right\}}

\newcommand{\trace}[1]{\textrm{tr}\left(#1\right)}

\newcommand{\Opw}[1]{\textrm{Op}^w_h\left({#1}\right)}

\theoremstyle{plain}
\newtheorem{theorem}[subsection]{Theorem}
\newtheorem{proposition}[subsection]{Proposition}

\theoremstyle{definition}

\theoremstyle{conjecture}

\theoremstyle{exercise}

\numberwithin{equation}{section}

\title{$L^p$-bounds for semigroups generated by non-elliptic quadratic differential operators}
\author{Francis White}
\affil{University of California Los Angeles}
\date{}

\begin{document}

\maketitle

\begin{abstract}
    In this note, we establish $L^p$-bounds for the semigroup $e^{-tq^w(x,D)}$, $t \ge 0$, generated by a quadratic differential operator $q^w(x,D)$ on $\R^n$ that is the Weyl quantization of a complex-valued quadratic form $q$ defined on the phase space $\R^{2n}$ with non-negative real part $\textrm{Re} \, q \ge 0$ and trivial singular space. Specifically, we show that $e^{-tq^w(x,D)}$ is bounded $L^p(\R^n) \rightarrow L^q(\R^n)$ for all $t > 0$ whenever $1 \le p \le q \le \infty$, and we prove bounds on $\norm{e^{-tq^w(x,D)}}_{L^p \rightarrow L^q}$ in both the large $t \gg 1$ and small $0 < t \ll 1$ time regimes. Regarding $L^p \rightarrow L^q$ bounds for the evolution semigroup at large times, we show that $\norm{e^{-tq^w(x,D)}}_{L^p \rightarrow L^q}$ is exponentially decaying as $t \rightarrow \infty$, and we determine the precise rate of exponential decay, which is independent of $(p,q)$. At small times $0 < t \ll 1$, we establish bounds on $\norm{e^{-tq^w(x,D)}}_{L^p \rightarrow L^q}$ for $(p,q)$ with $1 \le p \le q \le \infty$ that are polynomial in $t^{-1}$.
\end{abstract}

{\let\thefootnote\relax\footnote{{\emph{AMS Subject Classifications}}: 35Q40, 	35S30, 47D06, 47D08}}

\section{Introduction and Statement of Results}

In this note, we prove $L^p$-bounds for the solution operator $e^{-tq^w(x,D)}$ of the Schr\"{o}dinger initial value problem
\begin{align} \label{Schrodinger_initial_value_problem}
	\begin{cases}
		\p_t u(t,x) + q^w(x,D)u(t,x) = 0, \ \ (t,x) \in [0,\infty) \times \R^n, \\
		u(0,x) = u_0(x), \ \ x \in \R^n,
	\end{cases}
\end{align}
where $u_0 \in L^2(\R^n)$ is the initial data, $q = q(x,\xi)$ is a complex-valued quadratic form on the phase space $\R^{2n} = \R^n_x \times \R^n_\xi$ with non-negative real part $\textrm{Re} \ q \ge 0$, and $q^w(x,D)$ is the Weyl quantization of $q(x,\xi)$, defined by
\begin{align} \label{Weyl_quantization_quadratic_form}
	q^w(x,D)v(x) = (2\pi)^{-n} \int_{\R^n} \int_{\R^n} e^{i(x-y) \cdot \xi} q \left(\frac{x+y}{2}, \xi \right) v(y) \, dy \, d\xi, \ \ v \in \mathcal{S}'(\R^n),
\end{align}
in the sense of distributions. Operators of the form (\ref{Weyl_quantization_quadratic_form}) are quadratic differential operators with a simple, explicit expression. This is because the Weyl quantization of a quadratic monomial of the form $x^\alpha \xi^\beta$, where $\alpha, \beta \in \N^n$, $\abs{\alpha+\beta}=2$, is
\begin{align}
	\frac{x^\alpha D^\beta + D^\beta x^\alpha}{2}, \ \ D:=\frac{1}{i} \p.
\end{align}
The class of evolution equations of the form (\ref{Schrodinger_initial_value_problem}) contains a number of familiar examples, such as the free Schr\"{o}dinger equation where $q(x,\xi) = i \abs{\xi}^2$, $(x,\xi) \in \R^{2n}$, the quantum harmonic oscillator, where $q(x,\xi) = i(\abs{x}^2+\abs{\xi}^2)$, $(x,\xi) \in \R^{2n}$, the heat equation, where $q(x,\xi) = \abs{\xi}^2$, $(x,\xi) \in \R^{2n}$, and the Kramers-Fokker-Planck equation with a quadratic potential, where $q(x,v,\xi,\eta) = \eta^2+\frac{1}{4}v^2+i(v \cdot \xi-a x \cdot \eta)$, for $(x,v,\xi, \eta) \in \R^{4n} = \R^{2n}_{x,v} \times \R^{2n}_{\xi, \eta}$ and $a \in \R \backslash \{0\}$ a constant. From the work \cite{GeneralizedMehler}, it is known that the operator $q^w(x,D)$, regarded as an unbounded operator on $L^2(\R^n)$ equipped with the maximal domain
\begin{align} \label{maximal_domain}
	D_{\textrm{max}} = \set{u \in L^2(\R^n)}{q^w(x,D)u \in L^2(\R^n)},
\end{align}
is maximally accretive and generates a strongly continuous contraction semigroup $G(t):=e^{-tq^w(x,D)}$, $t \ge 0$, on $L^2(\R^n)$. We may regard $G(t)$ as the solution operator for the problem (\ref{Schrodinger_initial_value_problem}). Given that a wide range of physical processes may be modeled by equations of the form (\ref{Schrodinger_initial_value_problem}), it is of interest to understand the $L^p \rightarrow L^q$ mapping properties of the evolution semigroup $G(t)$ and to obtain bounds for the operator norm $\norm{G(t)}_{L^p \rightarrow L^q}$ at various time regimes. Let us mention that the study of $L^p$-bounds for semigroups generated by self-adjoint Schr\"{o}dinger operators has a long and rich tradition in the field of mathematical physics. We refer to \cite{SimonDavies84}, \cite{SimonDavies91}, \cite{Simon80}, and \cite{Simon81} for some fundamental results in this area. In particular, $L^p$-bounds for the propagator $G(t)$ were obtained in \cite{KochTataru} in the case when (\ref{Schrodinger_initial_value_problem}) is the time evolution of the quantum harmonic oscillator.

In this note, we shall be primarily interested in obtaining $L^p \rightarrow L^q$ bounds for $G(t)$ in the case when the quadratic form $q$ is non-elliptic. In order to recount the known results in this direction, we pause to recall the notion of the singular space of a complex-valued quadratic form $q$ on $\R^{2n}$ with non-negative real part $\textrm{Re} \ q \ge 0$. Let $\R^{2n}$ be equipped with the standard symplectic form
\begin{align} \label{Hamilton matrix of q}
    \sigma((x,\xi), (y,\eta)) = \xi \cdot y - x \cdot \eta, \ \ \ (x,\xi), \ (y,\eta) \in \R^{2n}.
\end{align}
Suppose $q: \R^{2n} \rightarrow \C$ is a complex-valued quadratic form with $\textrm{Re} \ q \ge 0$ and let $q(\cdot, \cdot)$ denote its symmetric $\C$-bilinear polarization. Because $\sigma$ is nondegenerate, there is a unique $F \in \textrm{Mat}_{2n \times 2n}(\C)$ such that
\begin{align} \label{def_Hamilton_matrix}
    q((x,\xi), (y,\eta)) = \sigma((x,\xi), F(y,\eta))
\end{align}
for all $(x,\xi), (y,\eta) \in \R^{2n}$. This matrix $F$ is called the \emph{Hamilton map} or \emph{Hamilton matrix of $q$} (see Section 21.5 of \cite{HormanderIII}). Explicitly, the Hamilton matrix of $q$ is given by
\begin{align}
	F = \frac{1}{2} H_{q},
\end{align}
where $H_q = \left(q'_{\xi}, -q'_x\right)$ is the Hamilton vector field of $q$, viewed as a linear map $\C^{2n} \rightarrow \C^{2n}$. Let
\begin{align*}
    \textrm{Re} \ F = \frac{F+\overline{F}}{2}, \ \ \textrm{Im} \ F = \frac{F-\overline{F}}{2i}
\end{align*}
be the real and imaginary parts of $F$ respectively. The \emph{singular space $S$ of $q$} is defined as the following finite intersection of kernels:
\begin{align} \label{definition of singular space}
    S = \left(\bigcap_{j=0}^{2n-1} \ker{\left[(\textrm{Re} \ F)(\textrm{Im} \ F)^j\right]} \right) \cap \R^{2n}.
\end{align}
The singular space was first introduced by M. Hitrik and K. Pravda-Starov in \cite{QuadraticOperators} where it arose naturally in the study of spectra and semi-group smoothing properties for non-self adjoint quadratic differential operators. The concept of the singular space has since been shown to play a key role in the understanding of hypoelliptic and spectral properties of non-elliptic quadratic differential operators. See for instance \cite{SemiclassicalHypoelliptic}, \cite{EigenvaluesAndSubelliptic}, \cite{ContractionSemigroup}, \cite{SubellipticEstimatesQuadraticDifferentialOperators}, \cite{NonEllipticQuadraticFormsandSemiclassicalEstimates}, and \cite{SpectralProjectionsAndResolventBounds}. Recent work has also shown that the singular space is vital for the description of the propagation of microlocal singularities for the evolution (\ref{Schrodinger_initial_value_problem}). We refer the reader to \cite{ExponentialSingularities}, \cite{time_dependent}, \cite{GaborSingularities}, \cite{PolynomialSingularities}, and \cite{Global_Analytic}, as well as \cite{PartialGS} and \cite{alphonse2020polar}.

Let $q$ be a complex-valued quadratic form on $\R^{2n}$ with non-negative real part $\textrm{Re} \, q \ge 0$. Let $S$ be the singular space of $q$. The quadratic form $q$ is said to be \emph{elliptic} if
\begin{align} \label{definition_ellipticity}
q(X) = 0, \ \ X \in \R^{2n} \implies X = 0,	
\end{align}
If (\ref{definition_ellipticity}) fails to hold, then we say that $q$ is \emph{non-elliptic}. To the best of our knowledge, there are currently only two general results regarding $L^p \rightarrow L^q$ bounds for the semigroup $G(t)$ in the case when $q$ is non-elliptic. First, in Theorem 1.2.3 of \cite{QuadraticOperators}, it was established that $\norm{G(t)}_{L^2 \rightarrow L^2}$ decays exponentially as $t \rightarrow \infty$ whenever $S$ is symplectic and distinct from the entire phase space $\R^{2n}$. In other words, if $S$ is symplectic and $S \neq \R^{2n}$, then there are $C,c>0$ such that
\begin{align} \label{L^2 exponential decay bound}
	\norm{G(t)}_{L^2 \rightarrow L^2} \le C e^{-ct}, \ \ t \ge 0.
\end{align}
Thanks to the subsequent work \cite{return2equilibrium}, it is also known that if $S$ is trivial, i.e. $S = \{0\}$, then the optimal rate of exponential decay of $\norm{G(t)}_{L^2 \rightarrow L^2}$ is the quantity $\gamma$ defined below in Theorem \ref{main_theorem}. The second general result concerning $L^p-L^q$ bounds for $G(t)$ is Theorem 1.2 of \cite{SubellipticEstimates}, which yields the following $L^2-L^\infty$ estimate: if $S=\{0\}$, then, for every $s>n/2$, there is $C>0$ such that
\begin{align} \label{L^2 L^infty bound}
	\norm{G(t)}_{L^2 \rightarrow L^\infty} \le C t^{-\frac{1}{2}(2k_0+1)(2n+s)}, \ \ 0<t \ll 1,
\end{align}
where $k_0 \in \{0,1, \ldots, 2n-1\}$ is the smallest non-negative integer such that
\begin{align} \label{def_k_0}
	\bigcap_{j=0}^{k_0} \ker{\left[(\textrm{Re} \ F)(\textrm{Im} \ F)^j\right]} \cap \R^{2n}	= \{0\}.
\end{align}
Our goal in the present work is to prove bounds for the operator norm $\norm{G(t)}_{L^p \rightarrow L^q}$ with $(p,q)$ more general than $(2,2)$ and $(2,\infty)$. The main result of this note refines and extends the bounds (\ref{L^2 exponential decay bound}) and (\ref{L^2 L^infty bound}) under the assumption that $S = \{0\}$. We recall from Theorem 1.2.2 of \cite{QuadraticOperators} that when $S=\{0\}$ the spectrum of the quadratic differential operator $q^w(x,D)$ is only composed of eigenvalues of finite algebraic multiplicity with
\begin{align} \label{explicit_spectrum}
	\textrm{Spec}(q^w(x,D)) = \set{\sum_{\substack{ \lambda \in \textrm{Spec}(F) \\ \textrm{Im}(\lambda) > 0}} (r_\lambda + 2k_\lambda)(-i \lambda)}{k_\lambda \in \N},
\end{align}
where $r_\lambda$ is the dimension of the space of generalized eigenvectors of the Hamilton matrix $F$ of $q$ in $\C^{2n}$ corresponding to the eigenvalue $\lambda \in \C$. In particular, the eigenvalue of $q^w(x,D)$ obtained by setting $k_\lambda = 0$ for all $\lambda \in \textrm{Spec}(F)$ in (\ref{explicit_spectrum}) is
\begin{align} \label{def_rho}
	\rho = \sum_{\substack{\lambda \in \textrm{Spec}(F) \\ \textrm{Im}(\lambda)>0}} -i r_\lambda \lambda.
\end{align}
We may think of $\rho$ as the `lowest eigenvalue' or `ground state energy' of the operator $q^w(x,D)$.

\begin{theorem} \label{main_theorem}
	Let $q$, $q^w(x,D)$, $G(t)$, $S$, and $F$ be as above. Assume that $S = \{0\}$.	
	\begin{enumerate}
	\item Let $\gamma = \textrm{Re}(\rho)>0$. For every $1 \le p \le q \le \infty$ and $\epsilon>0$, there are constants $C = C_{\epsilon, p, q}>0$ and $c = c_{p,q}>0$, such that
	\begin{align} \label{sharp_exponential_decay}
		c e^{-\gamma t} \le \norm{G(t)}_{L^p \rightarrow L^q} \le C e^{-\gamma t}, \ \ t \ge \epsilon.
	\end{align}
	\item Let $k_0 \in \{0,1, \ldots, 2n-1\}$ be the smallest non-negative integer such that (\ref{def_k_0}) holds. There is a time $0<t_0 \ll 1$ such that for any $1 \le p \le q \le \infty$ we have
	\begin{align} \label{short_time_bound}
		c \le \norm{G(t)}_{L^p \rightarrow L^q} \le C t^{-(2k_0+1)n}, \ \ 0 < t \le t_0,
	\end{align}
	for some constants $C=C_{p,q}>0$ and $c = c_{p,q}>0$. 
\end{enumerate}
\end{theorem}

Let us make some general comments regarding Theorem \ref{main_theorem} First, the bounds (\ref{sharp_exponential_decay}) show that for any $1 \le p \le q \le \infty$ the operator norm $\norm{G(t)}_{L^p \rightarrow L^q}$ decays exponentially as $t \rightarrow \infty$, with $\gamma$ being the precise rate of decay, independent of $(p,q)$. To prove that $\gamma$ is the exact rate of exponential decay, one may examine the action of the propagator $G(t)$ on the `ground state' eigenfunction of $q^w(x,D)$ corresponding to the eigenvalue $\rho$ (see Section 4 below). Regarding the short time $0<t \ll 1$ bounds in Theorem \ref{main_theorem}, it is clear that (\ref{short_time_bound}) is not sharp for all $1 \le p \le q \le \infty$. For instance, (\ref{short_time_bound}) fails to reproduce (\ref{L^2 exponential decay bound}) when $p=q=2$. However, one may interpolate (\ref{short_time_bound}) with the bound $G(t) = \mathcal{O}_{L^2 \rightarrow L^2}(1)$ as $t \rightarrow 0^+$ to obtain more precise estimates at short times. We also note that when $(p,q) = (2,\infty)$, the bound (\ref{short_time_bound}) gives $G(t) = \mathcal{O}_{L^2 \rightarrow L^\infty}(t^{-(2k_0+1)n})$ as $t \rightarrow 0^+$, which is an improvement over (\ref{L^2 L^infty bound}).

Finally, let us briefly touch on the main ideas involved in the proof of Theorem \ref{main_theorem}. In the recent work \cite{Global_Analytic}, we showed that if $\mathcal{T}_{\varphi}$ is a global metaplectic FBI transform on $\R^n$, in the sense of either Chapter 13 of \cite{SemiclassicalAnalysis} or the minicourse \cite{Minicourse}, then the conjugated propagator $\widetilde{G}(t):=\mathcal{T}_{\varphi} \circ G(t) \circ \mathcal{T}^*_{\varphi}$ is, for each $t \ge 0$, a metaplectic Fourier integral operator acting on the Bargmann space $H_{\Phi_0}(\C^n)$, which is the unitary image of $L^2(\R^n)$ under $\mathcal{T}_\varphi$. In particular, we showed that the `Bergman form' (\cite{ComplexFIOs}, \cite{semigroup2resolvent}) of $\widetilde{G}(t)$ is given by
\begin{align} \label{the_Bergman_form}
	\widetilde{G}(t)u(z) = \hat{a}(t) \int_{\C^n} e^{2 \Psi_t(z,\overline{w})} u(w) e^{-2 \Phi_0(w)} \, L(dw), \ \ z \in \C^n, \ \ u \in H_{\Phi_0}(\C^n), \ \ t \ge 0,
\end{align}
where $L(dw)$ is the Lebesgue measure on $\C^n$, $\Phi_0(w) := \sup_{y\in \R^n}(-\textrm{Im} \ \varphi(w,y)), \ w \in \C^n$, is the strictly plurisubharmonic quadratic form on $\C^n$ associated to $\varphi$, $\Psi_t$ is a holomorphic quadratic form on $\C^{2n} = \C^n \times \C^n$ depending analytically on $t \ge 0$, and $\hat{a} \in C^\omega([0,\infty); \C)$ is a non-vanishing amplitude. Moreover, we showed that $\Psi_t$ and $\hat{a}$ are the solutions of an eikonal equation and a transport equation, respectively. In particular, we did not attempt to solve these equations explicitly for $\Psi_t$ and $\hat{a}$. Now, thanks to the work \cite{quadratic_evol}, it is known that when the singular space is trivial $S=\{0\}$ it is possible to choose a metaplectic FBI transform $\mathcal{T}_\varphi$ so that conjugated semigroup has the simple form
\begin{align} \label{good_form_of_the_conjugated_propagator}
	\widetilde{G}(t)u(z) = e^{\frac{i}{2} \trace{M} t} u(e^{itM} z), \ \ u \in H_{\Phi_0}(\C^n), \ t \ge 0,
\end{align}
where $M \in \textrm{Mat}_{n \times n}(\C)$ is a suitable matrix. In the present work, we show that this choice of $\mathcal{T}_\varphi$ leads to equations for $\Psi_t$ and $\hat{a}$ that may be easily solved. One may then show that (\ref{the_Bergman_form}) coincides with (\ref{good_form_of_the_conjugated_propagator}), giving an alternative derivation of (\ref{good_form_of_the_conjugated_propagator}). Once the Bergman form of $\widetilde{G}(t)$ is known and a basic estimate for the real part of its phase function is established, the bounds (\ref{sharp_exponential_decay}) and (\ref{short_time_bound}) follow easily by writing down a formal expression for the Schwartz kernel of the composition $\mathcal{T}^*_{\varphi} \circ \widetilde{G}(t) \circ \mathcal{T}_{\varphi}$ using (\ref{the_Bergman_form}) and applying Young's integral inequality.

The plan for this note is as follows. In Section 2, we recall how to choose the FBI transform $\mathcal{T}_\varphi$ so that (\ref{good_form_of_the_conjugated_propagator}) holds. In Section 3, we determine the Bergman form (\ref{the_Bergman_form}) of $\widetilde{G}(t)$ for $t \ge 0$ and prove some basic estimates. In Section 4, we conclude the proof of Theorem \ref{main_theorem}, as outlined in this introduction.

\vspace{5mm}

{\bf{Acknowledgements.}} The author would like to express gratitude to Michael Hitrik for reading a preliminary draft of this manuscript and providing helpful feedback and suggestions. The author would also like to thank Daniel Parker for a stimulating conversation.

\section{Reduction to a Normal Form on the FBI Transform Side}

In this section, we follow the approach of \cite{resolvent_esimates_elliptic} and \cite{SpectralProjectionsAndResolventBounds} for reducing $q^w(x,D)$ to a normal form via a metaplectic FBI transform. We provide additional references where convenient.

Let $q$ be a complex-valued quadratic form on $\R^{2n}$ with non-negative real part $\textrm{Re} \ q \ge 0$ and trivial singular space $S = \{0\}$. Let $\C^{2n} = \C^n_z \times \C^n_\zeta$ be equipped with the standard complex symplectic form $\sigma = d\zeta \wedge dz$. Let $F$ be the Hamilton matrix of $q$ introduced in (\ref{def_Hamilton_matrix}). From the work \cite{QuadraticOperators}, it is known that the matrix $F$ has no real eigenvalues. Consequently,
\begin{align}
	\# \set{\lambda \in \textrm{Spec}(F)}{\textrm{Im} \, \lambda>0} = \# \set{\lambda \in \textrm{Spec}(F)}{\textrm{Im} \, \lambda<0},
\end{align}
counting algebraic multiplicities. For $\lambda \in \textrm{Spec}(F)$, let
\begin{align}
	V_\lambda = \ker{\left((F-\lambda)^{2n}\right)} \subset \C^{2n}
\end{align}
be the generalized eigenspace of $F$ corresponding to $\lambda$. Let us also introduce the stable outgoing and stable incoming manifolds for the quadratic form $-i q$ given by
\begin{align}
	\Lambda^+ = \bigoplus_{\substack{\lambda \in \textrm{Spec}(F) \\ \textrm{Im} \ \lambda>0}} V_\lambda, \ \ \Lambda^- = \bigoplus_{\substack{\lambda \in \textrm{Spec}(F) \\ \textrm{Im} \ \lambda<0}} V_\lambda,
\end{align}
respectively. By Proposition 2.1 of \cite{SpectralProjectionsAndResolventBounds}, $\Lambda^+$ is a strictly positive $\C$-Lagrangian subspace of $\C^{2n}$ in the sense that $\Lambda^+$ is Lagrangian with respect to the complex symplectic form $\sigma$ and
\begin{align} \label{positive_Lagrangian}
	\frac{1}{i} \sigma(Z, \overline{Z}) > 0, \ \ \ Z \in \Lambda^+ \backslash \{0\},
\end{align}
and $\Lambda^{-}$ is a strictly negative $\C$-Lagrangian subspace of $\C^{2n}$ in the sense that $\Lambda^{-}$ is Lagrangian for the form $\sigma$ and (\ref{positive_Lagrangian}) holds for all $Z \in \Lambda^{-} \backslash \{0\}$ with `$>$' replaced by `$<$'. For background information regarding positive and negative $\C$-Lagrangian subspaces of $\C^{2n}$, we refer to either \cite{Minicourse} or \cite{ComplexFIOs}. In particular, from the discussion on pages 488-489 of \cite{Minicourse}, we know that there exists a holomorphic quadratic form $\varphi = \varphi(z,y)$ on $\C^{2n} = \C^n_z \times \C^n_y$ with
\begin{align} \label{FBI phase function}
	\det{\varphi''_{zy}} \neq 0, \ \ \textrm{Im} \ \varphi''_{yy}>0,
\end{align}
such that the complex linear canonical transformation
\begin{align} \label{FBI canonical transformation}
	\kappa_{\varphi}: \C^{2n} \ni (y, -\varphi'_y(z,y)) \mapsto (z, \varphi'_z(z,y)) \in \C^{2n}, \ \ (z,y) \in \C^{2n},
\end{align}
generated by $\varphi$ satisfies
\begin{align} \label{image_of_lambdas}
	\kappa_\varphi(\Lambda^+) = \set{(z,0)}{z\in \C^n}, \ \ \ \kappa_\varphi(\Lambda^{-}) = \set{(0,\zeta)}{\zeta \in \C^n}.
\end{align}
Let
\begin{align} \label{definition_spsh_weight}
	\Phi_0(z) = \sup_{y\in \R^n} \left(-\textrm{Im} \ \varphi(z,y) \right), \ \ z \in \C^n,
\end{align}
be the strictly plurisubharmonic quadratic form on $\C^n$ associated to the phase $\varphi$ (see Chapter 13 of  \cite{SemiclassicalAnalysis} or Section 1.3 of \cite{Minicourse}), and let
\begin{align} \label{def_I_Lagrangian}
\Lambda_{\Phi_0} = \set{\left(z, \frac{2}{i} \Phi'_{0,z}(z)\right)}{z \in \C^n}.
\end{align}
From either Theorem 13.5 of \cite{SemiclassicalAnalysis} or Proposition 1.3.2 of \cite{Minicourse}, we have
\begin{align} \label{image real phase space}
	\kappa_\varphi \left(\R^{2n}\right) = \Lambda_{\Phi_0},
\end{align}
and thus $\Lambda_{\Phi_0}$ is $I$-Lagrangian and $R$-symplectic for the complex symplectic form $\sigma$. Also, the strict positivity of $\Lambda^+$ in conjunction with (\ref{image_of_lambdas}) gives that the base $\set{(z,0)}{z \in \C^n}$ is strictly positive relative to $\Lambda_{\Phi_0}$ (see e.g. \cite{ComplexFIOs}). It then follows, as explained in Chapter 11 of \cite{AnalyticMicrolocal_Analysis}, that the quadratic form $\Phi_0$ is strictly convex.

Let
\begin{align} \label{def q tilde}
	\tilde{q} = q \circ \kappa_\varphi^{-1},
\end{align}
regarded as a holomorphic quadratic form on $\C^{2n}$. Since $\Lambda^+$ and $\Lambda^-$ are invariant under $F$ and Lagrangian with respect to $\sigma$, we have
\begin{align}
	q(X) = \sigma(X,FX) = 0, \ \ X \in \Lambda^+ \cup \Lambda^-.
\end{align}
From (\ref{image_of_lambdas}) and (\ref{def q tilde}), it follows that $\tilde{q}$ must be of the form
\begin{align} \label{good_form}
	\tilde{q}(z,\zeta) = M z \cdot \zeta, \ \ (z,\zeta) \in \C^{2n},
\end{align}
for some $M \in \textrm{Mat}_{n \times n}(\C^n)$. In particular, the complex Hamilton vector field of $\tilde{q}$ with respect to $\sigma$ is
\begin{align} \label{Hamilton_field_good_form}
	H_{\tilde{q}} = \left(Mz, -M^T\zeta\right), \ \ (z,\zeta) \in \C^{2n}.
\end{align}
The Hamilton map of $\tilde{q}$ is thus given by $\tilde{F} = \frac{1}{2} H_{\tilde{q}}$, and we have
\begin{align}
	\tilde{F} = \frac{1}{2} \begin{pmatrix} M & 0 \\ 0 & -M^T \end{pmatrix}.
\end{align}
As a consequence of (\ref{def q tilde}), (\ref{def_Hamilton_matrix}), and the invariance of $\sigma$ under $\kappa_\varphi$, it is true that $\tilde{F} = \kappa_\varphi \circ F \circ \kappa_{\varphi}^{-1}$. Since also $\tilde{F}$ maps $(z,0) \in \kappa_\varphi(\Lambda^+)$ to $\frac{1}{2}(Mz,0) \in \kappa_\varphi(\Lambda^+)$, we have
\begin{align} \label{isospectrality_Hamilton_map}
	\textrm{Spec}(M) = \textrm{Spec}(2F) \cap \{\textrm{Im} \ \lambda>0\},
\end{align}
with agreement of algebraic multiplicities.

Let $\mathcal{T}_\varphi: \mathcal{S}'(\R^n) \rightarrow \textrm{Hol}(\C^n)$ be the metaplectic FBI transform on $\R^n$ associated to $\varphi$, given in the sense of distributions by
\begin{align} \label{FBI_transform_associated_to_phi}
	\mathcal{T}_\varphi u(z) = c_\varphi \int_{\R^n} e^{i \varphi(z,y)} u(y) \, L(dy), \ \ u \in \mathcal{S}'(\R^n),
\end{align}
where
\begin{align}
	c_\varphi = 2^{-n/2} \pi^{-3n/4} (\det{\textrm{Im} \ \varphi''_{yy}})^{-1/4} \abs{\det{\varphi''_{zy}}}.
\end{align}
By Theorem 13.7 of \cite{SemiclassicalAnalysis}, $\mathcal{T}_{\varphi}$ is unitary $L^2(\R^n) \rightarrow H_{\Phi_0}(\C^n)$, where
\begin{align} \label{the_Bargmann_space}
	H_{\Phi_0}(\C^n) := L^2(\C^n, e^{-2 \Phi_0(z)} \, L(dz)) \cap \textrm{Hol}(\C^n)
\end{align}
is the Bargmann space associated to the weight $\Phi_0$, equipped with the natural Hilbert space structure inherited from $L^2(\C^n, e^{-2 \Phi_0(z)} \, L(dz))$. Here $L(dz)$ denotes the Lebesgue measure on $\C^n$. Let $\tilde{q}^w(z,D)$ denote the complex Weyl quantization of the symbol $\tilde{q}$ with respect to the weight $\Phi_0$. We recall that $\tilde{q}^w(z,D)$ is defined as an unbounded operator on $H_{\Phi_0}(\C^n)$ that acts on suitable $u \in H_{\Phi_0}(\C^n)$ by
\begin{align} \label{complex_Weyl_quantization}
	\tilde{q}^w(z,D)u(z) = \frac{1}{(2\pi)^n} \iint_{\Gamma_{\Phi_0}(z)} e^{i(z-w) \cdot \zeta} \tilde{q}^w\left(\frac{z+w}{2}, \zeta \right) u(w) \, dw \wedge d\zeta, \ \ z \in \C^n,
\end{align}
for the contour of integration
\begin{align}
	\Gamma_{\Phi_0}(z): w \mapsto \zeta = \frac{2}{i} \Phi'_{0,z} \left(\frac{z+w}{2} \right), \ \ w \in \C^n, \ \ z \in \C^n.
\end{align}
For more information on Weyl quantization in the complex domain, see Chapter 13 of \cite{SemiclassicalAnalysis} or Section 1.4 of \cite{Minicourse}. By Egorov's theorem (see Theorem 13.9 in \cite{SemiclassicalAnalysis} or Theorem 1.4.2 of \cite{Minicourse}), we have
\begin{align} \label{Egorov_to_cite}
	q^w(x,D) = \mathcal{T}_{\varphi}^* \circ \tilde{q}^w(z,D) \circ \mathcal{T}_{\varphi}
\end{align}
when both sides are viewed as operators acting on the maximal domain of $q^w(x,D)$,
\begin{align} \label{the_maximal_domain}
	D_{\textrm{max}} = \set{u \in L^2(\R^n)}{q^w(x,D)u \in L^2(\R^n)}.
\end{align}
Let
\begin{align} \label{Bargmann_space_maximal_domain}
	\widetilde{D}_{\textrm{max}} = \set{u \in H_{\Phi_0}(\C^n)}{\tilde{q}^w(z,D)u \in H_{\Phi_0}(\C^n)}
\end{align}
be the maximal domain of $\tilde{q}^w(z,D)$, and let us view $\tilde{q}^w(z,D)$ as an unbounded operator on $H_{\Phi_0}(\C^n)$ with the domain $\widetilde{D}_{\textrm{max}}$. Thanks to (\ref{Egorov_to_cite}), we have
\begin{align} \label{equivalence_of_domains}
	\widetilde{D}_{\textrm{max}} = \mathcal{T}_{\varphi}(D_{\textrm{max}}).
\end{align}

Let $G(t) = e^{-tq^w(x,D)}$, $t \ge 0$, be the strongly continuous semigroup on $L^2(\R^n)$ generated by $q^w(x,D)$ (see \cite{GeneralizedMehler}). From (\ref{Egorov_to_cite}), (\ref{equivalence_of_domains}), and the unitarity of $\mathcal{T}_{\varphi}$, it follows that $\tilde{q}^w(z,D)$ generates a strongly continuous semigroup $\widetilde{G}(t) = e^{-t\tilde{q}^w(z,D)}$, $t \ge 0$, on $H_{\Phi_0}(\C^n)$. The semigroups $G(t)$ and $\widetilde{G}(t)$ are related by
\begin{align} \label{conjugated_the_semigroup}
	G(t) = \mathcal{T}^*_{\varphi} \circ \widetilde{G}(t) \circ \mathcal{T}_{\varphi}
\end{align}
for all $t \ge 0$. 

We have established the following proposition, which summarizes the discussion in this section.

\begin{proposition} \label{proposition_2.1}
	Let $q$ be a complex-valued quadratic form on $\R^{2n}$ with non-negative real part $\textrm{Re} \, q \ge 0$ and trivial singular space $S = \{0\}$. Let $F$ be the Hamilton matrix of $q$, and let $q^w(x,D)$ be the Weyl quantization of $q$, viewed as an unbounded operator on $L^2(\R^n)$ equipped with its maximal domain $D_{\textrm{max}}$ defined in (\ref{the_maximal_domain}).
	Let $G(t) = e^{-tq^w(x,D)}$, $t \ge 0$, be the strongly continuous semigroup on $L^2(\R^n)$ generated by $q^w(x,D)$.
	\begin{enumerate}
	\item 	There exists a holomorphic quadratic form $\varphi$ on $\C^{2n}$ satisfying (\ref{FBI phase function}) such that the quadratic form $\Phi_0$ defined by (\ref{definition_spsh_weight}) is strictly convex and the complex linear canonical transformation $\kappa_{\varphi}: \C^{2n} \rightarrow \C^{2n}$ defined implicitly by (\ref{FBI canonical transformation}) has the property that
	\begin{align}
		\tilde{q}(z,\zeta) := \left(q \circ \kappa_{\varphi}^{-1}\right)(z,\zeta) = M z \cdot \zeta, \ \ (z,\zeta) \in \C^{2n},
	\end{align}
	where $M \in \textrm{Mat}_{n \times n}(\C)$ is such that $\textrm{Spec}(M) = \textrm{Spec}(2F) \cap \{\textrm{Im} \ \lambda>0\}$ with agreement of algebraic multiplicities.
	\item 	Let $\tilde{q}^w(z,D)$ be the complex Weyl quantization (\ref{complex_Weyl_quantization}) of $\tilde{q}$ with respect to the weight $\Phi_0$, realized as an unbounded operator on the Bargmann space $H_{\Phi_0}(\C^n)$ introduced in (\ref{the_Bargmann_space}) equipped with the maximal domain $\widetilde{D}_{\textrm{max}}$ defined in (\ref{Bargmann_space_maximal_domain}). The operator $\tilde{q}^w(z,D)$ generates a strongly continuous semigroup $\widetilde{G}(t) = e^{-t\tilde{q}^w(z,D)}$, $t \ge 0$, on $H_{\Phi_0}(\C^n)$ that is unitarily equivalent to $G(t)$ for each $t \ge 0$. This unitary equivalence is given by the FBI transform $\mathcal{T}_\varphi$ introduced in (\ref{FBI_transform_associated_to_phi}), i.e.
	\begin{align}
		G(t) = \mathcal{T}_{\varphi}^* \circ \widetilde{G}(t) \circ \mathcal{T}_\varphi, \ \ t \ge 0.
	\end{align}
	\end{enumerate}
\end{proposition}

\section{The Evolution Semigroup on the FBI Transform Side}

We now study the semigroup $\widetilde{G}(t)$, $t \ge 0$. Let $\Psi_0$ be the polarization of $\Phi_0$, i.e. $\Psi_{0}$ is the unique holomorphic quadratic form on $\C^{2n} = \C^n \times \C^n$ such that $\Psi_{0}(z, \overline{z}) = \Phi_0(z)$ for all $z \in \C^n$. Since
\begin{align} \label{full_form_weight}
	\Phi_0(z) = \frac{1}{2} \Phi''_{0, zz} z \cdot z + \Phi''_{0, \overline{z} z} z \cdot \overline{z} + \frac{1}{2} \Phi''_{0, \overline{z} \, \overline{z}}  \overline{z} \cdot \overline{z}, \ \ z \in \C^n,
\end{align}
we see that $\Psi_0$ is given explicitly by
\begin{align} \label{explicit_polarization}
	\Psi_0(z,\theta) = \frac{1}{2} \Phi''_{0, zz} z \cdot z + \Phi''_{0, \overline{z} z} z \cdot \theta + \frac{1}{2} \Phi''_{0, \overline{z} \, \overline{z}} \theta \cdot \theta, \ \ (z,\theta) \in \C^{2n}.
\end{align}
In the work \cite{Global_Analytic}, we showed that for every $t \ge 0$ the semigroup $\widetilde{G}(t)$ is a metaplectic Fourier integral operator in the complex domain whose underlying complex canonical transformation is the Hamilton flow $\tilde{\kappa}_t$ of the symbol $\tilde{q}$ at time $t/i$, i.e.
\begin{align} \label{def_Hamilton_flow}
	\tilde{\kappa}_t = \exp{\left(\frac{t}{i} H_{\tilde{q}} \right)}, \ \ t \ge 0.
\end{align}
In view of (\ref{Hamilton_field_good_form}), we have
\begin{align} \label{explicit_Hamilton_flow}
	\tilde{\kappa}_t(z,\zeta) = \left(e^{-itM} z, e^{itM^T} \zeta \right), \ \ (z,\zeta) \in \C^{2n}, \ t \ge 0.
\end{align}
For background information regarding metaplectic Fourier integral operators in the complex domain, see Appendix B of \cite{PTSymmetric}. In particular, in the work \cite{ComplexFIOs}, it was shown that every such metaplectic Fourier integral operator in $\C^n$ possesses a unique `Bergman form.' In Section 6 of \cite{Global_Analytic}, we proved that the Bergman form of $\tilde{G}(t)$ is given by
\begin{align} \label{the_good_Bergman_form}
	\widetilde{G}(t)u(z) = \hat{a}(t) \int_{\C^n} e^{2 \Psi_t(z,\overline{w})} u(w) e^{-2 \Phi_0(w)} \, L(dw), \ \ z \in \C^n, \ \ u \in H_{\Phi_0}(\C^n),
\end{align}
where $\Psi_t$ is a holomorphic quadratic form on $\C^{2n}$, depending analytically on $t \ge 0$, and $\hat{a} \in C^\omega([0,\infty); \C)$ is a non-vanishing amplitude. In addition, we showed that $\Psi_t$, $t \ge 0$, is the unique solution of the eikonal equation
\begin{align} \label{eikonal_equation_phase}
	\begin{cases}
		2 \p_t \Psi_t(z,\theta) + \tilde{q} \left(z, \frac{2}{i} \Psi'_{t,z}(z,\theta) \right)=0, \ \ (z,\theta) \in \C^{2n}, \ \ t \ge 0, \\
		\left. \Psi_t(z,\theta) \right|_{t=0} = \Psi_{0}(z,\theta), \ \ (z,\theta) \in \C^{2n},
	\end{cases}
\end{align}
and $\hat{a}$ is the unique solution of the transport equation
\begin{align} \label{transport_equation}
	\begin{cases}
		\hat{a}'(t) + \frac{1}{2i} \beta(t) \hat{a}(t) = 0, \ \ t \ge 0, \\
		\hat{a}(0) = C_{\Phi_0},
	\end{cases}
\end{align}
where
\begin{align}
	\beta(t) = \trace{\tilde{q}''_{\zeta z} + \tilde{q}''_{\zeta \zeta} \cdot \frac{2}{i} \Psi''_{t,zz}}, \ \ t \ge 0,
\end{align}
and
\begin{align}
	C_{\Phi_0} = 2^n \pi^{-n} \det{\Phi_{0, z\overline{z}}''}.
\end{align}
We note that the initial conditions in (\ref{eikonal_equation_phase}) and (\ref{transport_equation}) are chosen so that when $t=0$ the righthand side of (\ref{the_good_Bergman_form}) coincides with the orthogonal projector $\Pi_{\Phi_0}: L^2(\C^n, e^{-2 \Phi_0(z)} \, L(dz)) \rightarrow H_{\Phi_0}(\C^n)$, which has the explicit integral represenation
\begin{align} \label{Bergman_projector}
	\Pi_{\Phi_0} u(z) = C_{\Phi_0} \int_{\C^n} e^{2 \Psi_0(z, \overline{w})} u(w) e^{-2 \Phi_0(w)} \, L(dw), \ \ u \in L^2 (\C^n, e^{-2 \Phi(z)} \, L(dz)).
\end{align}
In the literature, the operator $\Pi_{\Phi_0}$ is known as the `Bergman projector' associated to the weight $\Phi_0$. For a proof of (\ref{Bergman_projector}), see Theorem 13.6 of \cite{SemiclassicalAnalysis} or Proposition 1.3.4 of \cite{Minicourse}.

Since $\tilde{q}$ has the simple form (\ref{good_form}),  we may determine $\Psi_t$ and $\hat{a}$ by solving (\ref{eikonal_equation_phase}) and (\ref{transport_equation}) explicitly. We begin by studying the transport equation (\ref{transport_equation}). Thanks to (\ref{good_form}), we see that
\begin{align}
	\beta(t) = \trace{M}, \ \ t \ge 0.
\end{align}
The unique solution of (\ref{transport_equation}) is
\begin{align} \label{solution_transport_equation}
	\hat{a}(t) = C_{\Phi_0} e^{\frac{i}{2} \trace{M} t}, \ \ t \ge 0.
\end{align}
Next, we solve (\ref{eikonal_equation_phase}) for $\Psi_t$. We search for a solution to (\ref{eikonal_equation_phase}) of the form
\begin{align} \label{ansatz_for_eikonal_equation}
	\Psi_t(z,\theta) = \frac{1}{2} A_t z \cdot z + B_t z \cdot \theta + \frac{1}{2} D_t \theta \cdot \theta, \ \ (z,\theta) \in \C^{2n}, \ \ t \ge 0,
\end{align}
where $A_t, B_t, D_t \in \textrm{Mat}_{n \times n}(\C)$ depend smoothly on $t$ and $A_t = A_t^T$ and $D_t = D_t^T$ for all $t \ge 0$.  Inserting (\ref{ansatz_for_eikonal_equation}) into (\ref{eikonal_equation_phase}) and using (\ref{good_form}) and (\ref{explicit_polarization}), we see that $\Psi_t$ will be a solution of (\ref{eikonal_equation_phase}) provided $A_t, B_t$, and $D_t$ satisfy
\begin{align} \label{first_initial_value_problem}
	\begin{cases}
		\p_t A_t z \cdot z + \frac{2}{i} A_t Mz \cdot z = 0, \ \ z \in \C^n, \ t \ge 0, \\
		A_0 = \Phi''_{0, zz},
	\end{cases}
\end{align}
\begin{align} \label{second_initial_value_problem}
	\begin{cases}
	\p_t B_t z \cdot \theta + \frac{1}{i} B_t M z \cdot \theta = 0, \ \ z, \theta \in \C^n, \ t \ge 0, \\
	B_0 = \Phi''_{0, \overline{z} z},
	\end{cases}
\end{align}
and
\begin{align} \label{third_initial_value_problem}
	\begin{cases}
		\p_t C_t \theta \cdot \theta = 0, \ \ \theta \in \C^n, \ t \ge 0, \\
		C_0 = \Phi''_{0, \overline{z} \overline{z}},
	\end{cases}
\end{align}
respectively. The symmetry of $A_t$ implies that
\begin{align}
	2 A_t M z \cdot z = \left(A_t M + M^T A_t \right) z \cdot z, \ \ z \in \C^n, \ t \ge 0.
\end{align}
Thus (\ref{first_initial_value_problem}) holds if and only if
\begin{align} \label{better_initial_value_problem}
	\begin{cases}
		\p_t A_t + \frac{1}{i} A_t M + \frac{1}{i} M^T A_t = 0, \ \ t \ge 0, \\
		A_0 = \Phi''_{0, zz}. 
	\end{cases}
\end{align}
The unique solution of (\ref{better_initial_value_problem}) is
\begin{align}
	A_t = e^{i M^T t} \Phi''_{0, zz} e^{i M t}, \ \ t \ge 0.
\end{align}
By inspection, the solutions of (\ref{second_initial_value_problem}) and (\ref{third_initial_value_problem}) are
\begin{align}
	B_t = \Phi''_{0, \overline{z} z} e^{i t M}, \ C_t = \Phi''_{0, \overline{z} \overline{z}}, \ \ t \ge 0,
\end{align}
respectively. Using (\ref{explicit_polarization}), we get
\begin{align} \label{explicit_form_psi_t}
	\Psi_t(z,\theta) = \Psi_0(e^{itM} z, \theta), \ \ (z,\theta) \in \C^{2n}, \ t \ge 0.
\end{align}
From (\ref{the_good_Bergman_form}), (\ref{Bergman_projector}), (\ref{solution_transport_equation}), and (\ref{explicit_form_psi_t}), we deduce that
\begin{align} \label{propagator_is_pullback}
	\widetilde{G}(t) u(z) = e^{\frac{i}{2} \trace{M} t} u \left(e^{itM} z \right), \ \ u \in H_{\Phi_0}(\C^n), \ t \ge 0.
\end{align}
The formula (\ref{propagator_is_pullback}) for the semigroup $\widetilde{G}(t)$ was obtained by a different method in \cite{quadratic_evol}.

For $t \ge 0$, let us define
\begin{align} \label{explicit_form_phi_t}
	\Phi_t(z) = \Phi_0 \left(e^{itM} z \right), \ \ z \in \C^n, \ t \ge 0.
\end{align}
Since $\Phi_0$ is strictly convex, $\Phi_t$ is a strictly convex quadratic form on $\C^n$ for all $t \ge 0$. In addition, we have $\Phi_{t}|_{t=0} = \Phi_0$. For $t \ge 0$, let
\begin{align} \label{t-dependent_Bargmann_space}
	H_{\Phi_t}(\C^n) = L^2 \left(\C^n, e^{-2 \Phi_t(z)} \, L(dz) \right) \cap \textrm{Hol}(\C^n)
\end{align}
be the Bargmann space associated to $\Phi_t$, equipped with the natural Hilbert space structure induced from $L^2 \left(\C^n, e^{-2 \Phi_t(z)} \, L(dz) \right)$. From (\ref{propagator_is_pullback}), it is clear that $\widetilde{G}(t)$ is bounded $H_{\Phi_0}(\C^n) \rightarrow H_{\Phi_t}(\C^n)$ for every $t \ge 0$, and a direct computation using (\ref{propagator_is_pullback}), (\ref{explicit_form_phi_t}), and (\ref{isospectrality_Hamilton_map}) gives
\begin{align}
	\norm{\widetilde{G}(t)u}_{H_{\Phi_t}(\C^n)} = e^{\gamma t} \norm{u}_{H_{\Phi_0}(\C^n)}, \ \ u \in H_{\Phi_0}(\C^n), \ t \ge 0,
\end{align}
where $\gamma>0$ is as in the statement of Theorem \ref{main_theorem}.

The following proposition summarizes the discussion so far in this section and establishes some basic estimates that will be necessary for the proof of Theorem \ref{main_theorem} in Section 4.

\begin{proposition} \label{Proposition_3.1}
Let $q$, $\tilde{q}$, $M$, $\Phi_0$, $H_{\Phi_0}(\C^n)$, and $\widetilde{G}(t)$ be as in Proposition \ref{proposition_2.1}.
	\begin{enumerate}
		\item For every $t \ge 0$, we have
	\begin{align}
		\widetilde{G}(t) u(z) = e^{\frac{i}{2} \trace{M} t} u \left(e^{itM} z \right), \ \ u \in H_{\Phi_0}(\C^n).
	\end{align}
	In addition,
	\begin{align}
		\norm{\widetilde{G}(t) u}_{H_{\Phi_t}(\C^n)} = e^{\gamma t}\norm{u}_{H_{\Phi_0}(\C^n)}, \ \ t \ge 0,
	\end{align}
	where
	\begin{align} \label{propagated_weights}
		\Phi_t(z) = \Phi_0 \left(e^{itM} z \right), \ \ z \in \C^n, \ t \ge 0,
	\end{align}
	the norm $\norm{\cdot}_{H_{\Phi_t}(\C^n)}$ is the norm on the Bargmann space $H_{\Phi_t}(\C^n)$ introduced in (\ref{t-dependent_Bargmann_space}), and $\gamma>0$ is as in the statement of Theorem 1.1.
	
	\item Let $R_t = \Phi_0 - \Phi_t$, $t \ge 0$, and let $\alpha: [0, \infty) \rightarrow \R$ be the continuous function defined by
	\begin{align} \label{def_alpha_t}
		\alpha(t) = \textrm{min}_{\abs{z}=1} R_t(z),
	\end{align}
	so that
	\begin{align} \label{bound_for_R_t}
	R_t(z) \ge \alpha(t) \abs{z}^2, \ \ z \in \C^n, \ t \ge 0.	
	\end{align}
	The function $\alpha$ has the following properties:
	\begin{enumerate}
	\item $\alpha(0) = 0$ and $\alpha(t)>0$ for all $t>0$,
	\item $\alpha$ is non-decreasing,
	\item there is $0<t_0 \ll 1$ and $c>0$ such that
	\begin{align} \label{short_time_lower_bound}
		\alpha(t) \ge c t^{2k_0+1}, \ \ 0 \le t \le t_0,
	\end{align}
	where $k_0 \in \{0,1,\ldots, 2n-1\}$ is the smallest non-negative integer such that (\ref{def_k_0}) holds, and
	\item $\alpha(t) \rightarrow \min_{\abs{z}=1} \Phi_0(z)>0$ as $t \rightarrow \infty$.
	\end{enumerate}

	\item Let $\Psi_0$ be the polarization of $\Phi_0$ given by (\ref{explicit_polarization}). For any $t \ge 0$ and $u \in H_{\Phi_0}(\C^n)$, we have
	\begin{align}
		\widetilde{G}(t)u(z) = C_{\Phi_0} e^{\frac{i}{2} \trace{M} t} \int_{\C^n} e^{2 \Psi_t(z,\overline{w})} u(w) e^{-2 \Phi_0(w)} \, L(dw), \ \ z \in \C^n,
	\end{align}
	where
	\begin{align}
		\Psi_t(z,\theta) = \Psi_0 \left(e^{itM} z, \theta \right), \ \ (z,\theta) \in \C^{2n}, \ t \ge 0.
	\end{align}
	Moreover, there are constants $C,c>0$, independent of $t$, such that
	\begin{align} \label{fundamental_estimate}
			-C \abs{w - e^{it M} z}^2 \le 2 \, \textrm{Re} \, \Psi_t(z,\overline{w}) - \Phi_t(z) - \Phi_0(w) \le -c \abs{w-e^{itM} z}^2, \ \ z,w \in \C^n, \ t \ge 0.
	\end{align}

	\end{enumerate}

\end{proposition}
\begin{proof}
It remains to establish Point 2 and the estimate (\ref{fundamental_estimate}). To this end, let
\begin{align}
	R_t(z) = \Phi_0(z) - \Phi_t(z), \ \ z \in \C^n, \ t \ge 0,
\end{align}
and let $\alpha: [0,\infty) \rightarrow \infty$ be as in (\ref{def_alpha_t}). We will begin by showing that
\begin{align} \label{domination_weights}
	R_t(z) \ge 0, \ \ z \in \C^n, \ t \ge 0.
\end{align}
Let $\tilde{\kappa}_t$, $t \ge 0$, be as in (\ref{def_Hamilton_flow}). A straightforward computation using (\ref{def_I_Lagrangian}), (\ref{explicit_Hamilton_flow}), and (\ref{propagated_weights}) gives that
\begin{align} \label{mapping_I_Lagrangians}
	\tilde{\kappa}_t \left(\Lambda_{\Phi_0} \right) = \Lambda_{\Phi_t} := \set{\left(z, \frac{2}{i} \Phi'_{t,z}(z) \right)}{z \in \C^n}, \ \ t \ge 0.
\end{align}
From either the discussion in Section 6 of \cite{Global_Analytic} or a direct computation, we know that the family $(\Phi_t)_{t \ge 0}$ satisfies the eikonal equation
\begin{align}
	\begin{cases}
		\p_t \Phi_t(z) + \textrm{Re} \, \tilde{q}\left(z, \frac{2}{i} \Phi'_{t,z}(z) \right) = 0, \ \ z \in \C^n, \ t \ge 0, \\
		\left. \Phi_t \right|_{t=0} = \Phi_0 \ \textrm{on} \ \C^n.
	\end{cases}
\end{align}
As a consequence of (\ref{mapping_I_Lagrangians}), for every $z \in \C^n$ and $t \ge 0$, there is a point $Z \in \Lambda_{\Phi_0}$ such that
\begin{align}
	\left(z, \frac{2}{i} \Phi'_{t,z}(z) \right) = \tilde{\kappa}_t (Z). 
\end{align}
Since $\tilde{q}$ is invariant under the flow $\tilde{\kappa}_t$, for every $t \ge 0$ and $z \in \C^n$, there is $Z \in \Lambda_{\Phi_0}$ such that
\begin{align}
	\p_t \Phi_t(z) = -\textrm{Re} \, \tilde{q}(Z).
\end{align}
Because $\textrm{Re} \, q \ge 0$, (\ref{image real phase space}) and (\ref{def q tilde}) imply that $\textrm{Re} \, \tilde{q} \ge 0$ on $\Lambda_{\Phi_{0}}$, and we have
\begin{align}
	\p_t \Phi_t(z) \le 0, \ \ z \in \C^n, \ t \ge 0.
\end{align}
Thus, for any fixed $z \in \C^n$, the function
\begin{align}
	t \mapsto \Phi_0(z) - \Phi_t(z)
\end{align}
is non-decreasing. It follows that $R_t \ge 0$ for all $t \ge 0$ and that the function $\alpha$ is non-decreasing.

We next recall from Proposition 6.1 of \cite{Global_Analytic} that
\begin{align} \label{FBI_side_characterization_singular_space}
	\Lambda_{\Phi_0} \cap \Lambda_{\Phi_t} = \pi_1 \left(\kappa_{\varphi}(S) \right), \ \ t > 0,
\end{align}
where $S$ is the singular space of $q$, $\kappa_{\varphi}: \C^{2n} \rightarrow \C^{2n}$ is the complex linear canonical transformation defined by (\ref{FBI canonical transformation}), and $\pi_1: \C^{2n} \rightarrow \C^n$ is the projection $\pi_1: (z,\zeta) \mapsto z$. Since we assume that $S = \{0\}$, we deduce from (\ref{FBI_side_characterization_singular_space}) that
\begin{align}
	\Lambda_{\Phi_0} \cap \Lambda_{\Phi_t} = \{0\}, \ \ t>0.
\end{align}
Thus, for every $t>0$ and $z \in \C^n$,
\begin{align}
	\frac{2}{i} \Phi'_{0, z}(z) - \frac{2}{i} \Phi'_{t,z}(z) = 0 \iff z = 0.
\end{align}
Because $R_t$ is a non-negative quadratic form for each $t \ge 0$, we have
\begin{align}
	R_t(z) = 0, \ z \in \C^n, \ t>0 \iff \nabla_{\textrm{Re} \, z, \textrm{Im} \, z}  R_t(z) = 0 \iff \frac{2}{i} \Phi'_{0, z}(z) - \frac{2}{i} \Phi'_{t,z}(z) = 0.
\end{align}
Hence, for any $z \in \C^n$ and $t>0$,
\begin{align}
	R_t(z) = 0 \iff z = 0.
\end{align}
Thus $\alpha(t)>0$ for all $t>0$.

To establish (\ref{short_time_lower_bound}),  we recall the main result of Section 2 of \cite{SubellipticEstimates}, which states that if the singular space of $q$ is trivial, $S = \{0\}$, then there is a small time $0<t_0 \ll 1$ and a constant $c>0$ such that
\begin{align}
	R_t(z) \ge c t^{2k_0+1} \abs{z}^2, \ \ z \in \C^n, \ 0 \le t \le t_0,
\end{align}
where $k_0 \in \{0,1, \ldots, 2n-1\}$ is the smallest non-negative integer such that (\ref{def_k_0}) holds. It is therefore true that
\begin{align}
	\alpha(t) \ge c t^{2k_0+1}, \ \ 0 \le t \le t_0.
\end{align}

To prove the claim regarding the behavior of $\alpha(t)$ as $t \rightarrow \infty$, we note that (\ref{isospectrality_Hamilton_map}) implies that $\textrm{spec}(iM) \subset \{\textrm{Re} \, \lambda<0\}$. Thus there is $c>0$ such that
\begin{align}
	R_t(z) = \Phi_0(z) + \mathcal{O}(e^{-ct} \abs{z}^2) \ \textrm{as} \ t \rightarrow \infty.
\end{align}
It follows that
\begin{align}
	\alpha(t) \rightarrow \textrm{min}_{\abs{z}=1} \Phi_0(z) \ \textrm{as} \ t \rightarrow \infty.
\end{align}
The proof of Point 2 is complete.
 
Finally, we prove (\ref{fundamental_estimate}). Using (\ref{full_form_weight}), (\ref{explicit_polarization}), (\ref{explicit_form_psi_t}), and (\ref{explicit_form_phi_t}), we obtain the following identity by elementary algebraic manipulations:
\begin{align}
\begin{split}
	2 \, \textrm{Re} \, \Psi_t(z,\overline{w}) - \Phi_t(z) - \Phi_0(w) = -\Phi''_{0, \overline{z} z} \left(w - e^{i M t} z \right) \cdot \overline{\left(w - e^{i M t} z \right)}, \ \ z, w \in \C^n, \ t \ge 0.
\end{split}
\end{align}
Because $\Phi_0$ is a strictly plurisubharmonic quadratic form, the Levi matrix $\Phi''_{0, \overline{z} z}$ is Hermitian positive-definite. Consequently, there are constants $C, c>0$, independent of $t$, such that
\begin{align}
	-C \abs{w - e^{it M} z}^2 \le 2 \, \textrm{Re} \, \Psi_t(z,\overline{w}) - \Phi_t(z) - \Phi_0(w) \le -c \abs{w-e^{itM} z}^2, \ \ z,w \in \C^n, \ t \ge 0.
\end{align}
This proves (\ref{fundamental_estimate}).
\end{proof}

\section{The Conclusion of the Proof of Theorem \ref{main_theorem}}

In view of (\ref{conjugated_the_semigroup}), (\ref{FBI_transform_associated_to_phi}), (\ref{the_good_Bergman_form}), and (\ref{solution_transport_equation}), the Schwartz kernel $K_t(x,y)$ of $G(t)$ is given, formally, by
\begin{align} \label{expression_Schwartz_kernel}
	K_t(x,y) = c_\varphi^2 C_{\Phi_0} e^{\frac{i}{2}\trace{M}t} \int_{\C^n} \int_{\C^n} e^{P_t(x,y,z,w)} \, L(dw) \, L(dz), \ \ (x,y) \in \R^{2n}, \ \ t \ge 0,
\end{align}
where
\begin{align}
	P_t(x,y,z,w) := -i \overline{\varphi(z,x)}-2\Phi_0(z)+2 \Psi_t(z,\overline{w})-2\Phi_0(w)+i\varphi(w,y),
\end{align}
for $x,y \in \R^n$, $z,w \in \C^n$, and $t \ge 0$. For $z \in \C^n$, let $r(z) \in \R^n$ be the unique point such that
\begin{align}
	\Phi_0(z) = -\textrm{Im} \ \varphi(z, r(z)).
\end{align}
Since $\textrm{Im} \ \varphi''_{yy}>0$, there is $c>0$ such that
\begin{align} \label{estimate_FBI_phase}
	-\textrm{Im} \ \varphi(z,y) - \Phi_0(z) \le -c \abs{y-r(z)}^2, \ \ z \in \C^n, \ y \in \R^n.
\end{align}
Using (\ref{estimate_FBI_phase}) together with the estimate (\ref{fundamental_estimate}), we find that
\begin{align} \label{intermediate_estimate_P}
	\textrm{Re} \ P_t(x,y,z,w) \le -c \abs{x-r(z)}^2 - R_t(z) - c \abs{w-e^{itM} z}^2-c\abs{y-r(w)}^2,
\end{align}
for all $x,y \in \R^n$, $z,w \in \C^n$, and $t \ge 0$, where $R_t(z)$ is as in Proposition \ref{Proposition_3.1}. Let $\alpha: [0,\infty) \rightarrow \R$ be as in (\ref{def_alpha_t}). Since (\ref{bound_for_R_t}) holds, there is $c>0$ such that
\begin{align} \label{general_bound_phase_function}
	\textrm{Re} \, P_t(x,y,z,w) \le -c \abs{x-r(z)}^2-\alpha(t) \abs{z}^2 - c \abs{w - e^{itM}z}^2 - c \abs{y-r(w)}^2
\end{align}
for all $x,y \in \R^n$, $z,w \in \C^n$, and $t \ge 0$. 

Let $\gamma$ be as in the statement of Theorem \ref{main_theorem}. Taking the absolute value of (\ref{expression_Schwartz_kernel}) and using (\ref{general_bound_phase_function}) and (\ref{isospectrality_Hamilton_map}), we find that there are constants $C, c>0$ such that
\begin{align} \label{pointwise_bound_Schwartz_kernel}
	\abs{K_t(x,y)} \le C e^{-\gamma t} \int_{\C^n} \int_{\C^n} e^{-c \abs{x-r(z)}^2-\alpha(t) \abs{z}^2 - c \abs{w - \exp{(itM)}z}^2 - c \abs{y-r(w)}^2} \, L(dw) \, L(dz)
\end{align}
for every $x,y \in \R^n$ and $t \ge 0$. Let $1 \le p \le q \le \infty$ be given, and let $1 \le r \le \infty$ be such that
\begin{align}
	1+\frac{1}{q} = \frac{1}{p}+\frac{1}{r}.
\end{align}
Using Minkowski's integral inequality and the fact that $\alpha(t)>0$ for every $t>0$, we get that
\begin{align} \label{first_bound_integral_inequality}
\begin{split}
	\norm{K_t(x,\cdot)}_{L^r} &\le C e^{-\gamma t} \int_{\C^n} \int_{\C^n} e^{-c \abs{x-r(z)}^2-\alpha(t) \abs{z}^2-c \abs{w-\exp{(itM)}z}^2} \norm{e^{-c \abs{y-r(w)}^2}}_{L^r_y} \, L(dw) \, L(dz) \\
	&\le C \alpha(t)^{-n} e^{-\gamma t}, \ \ x \in \R^n, \ t>0,
\end{split}
\end{align}
where $C=C_{p,q}>0$ depends only on $p$ and $q$. By similar reasoning, there is $C=C_{p,q}>0$ such that
\begin{align} \label{second_bound_integral_inequality}
	\norm{K_t(\cdot, y)}_{L^r} \le C \alpha(t)^{-n} e^{-\gamma t}, \ \ y \in \R^n, \ t>0.
\end{align}
Applying Young's integral inequality with (\ref{first_bound_integral_inequality}) and (\ref{second_bound_integral_inequality}) gives
\begin{align} \label{bound_for_G(t)}
	\norm{G(t)}_{L^p \rightarrow L^q} \le C \alpha(t)^{-n} e^{-\gamma t}, \ \ t>0,
\end{align}
for some $C = C_{p,q}>0$.

Let $\epsilon>0$ be arbitrary. From Proposition \ref{Proposition_3.1}, we know that $\alpha$ is non-decreasing and $\alpha(t)>0$ for all $t>0$. Thus,
\begin{align}
	\alpha(t) \ge \alpha(\epsilon), \ \ t \ge \epsilon.
\end{align}
In view of (\ref{bound_for_G(t)}), we may deduce that there is $C=C_{\epsilon, p, q}>0$ such that
\begin{align} \label{exponential_decay_bound}
	\norm{G(t)}_{L^p \rightarrow L^q} \le C e^{-\gamma t}, \ \ t \ge \epsilon.
\end{align}
To see that the bound (\ref{exponential_decay_bound}) is sharp as $t \rightarrow \infty$, we recall from Theorem 2.1 of \cite{return2equilibrium} that the lowest eigenvalue $\rho$ of $q^w(x,D)$, introduced in (\ref{def_rho}), is simple and that the eigenspace of $q^w(x,D)$ corresponding to $\rho$ is spanned by a `ground state' of the form
\begin{align}
	u_0(x) = e^{-a(x)}, \ \ x \in \R^n,
\end{align}
where $a$ is a complex-valued quadratic form on $\R^n$ with positive-definite real part $\textrm{Re} \, a >0$. Let $v = \norm{u_0}_{L^p(\R^n)}^{-1} u_0$. Since $q^w(x,D) v = \rho v$, 
is is clear that
\begin{align}
	\norm{e^{-tq^w(x,D)} v}_{L^q} = e^{-t \gamma } \norm{v}_{L^q}, \ \ t \ge 0.
\end{align}
Hence there is a constant $c=c_{p,q}>0$ such that
\begin{align} \label{proving_sharpness}
	\norm{e^{-tq^w(x,D)}}_{L^p \rightarrow L^q} \ge c e^{-\gamma t}, \ \ t \ge 0.
\end{align}
We conclude that there are constants $C = C_{\epsilon, p,q}>0$ and $c = c_{p,q}>0$ such that (\ref{sharp_exponential_decay}) holds for all $t \ge \epsilon$.

Finally, we prove the bound (\ref{short_time_bound}). From (\ref{short_time_lower_bound}), (\ref{bound_for_G(t)}), and (\ref{proving_sharpness}), we get that there are constants $C = C _{p,q}>0$ and $c_{p,q}>0$ such that
\begin{align} \label{really course bound}
	c \le \norm{G(t)}_{L^p \rightarrow L^q} \le C t^{-(2k_0+1)n}, \ \ 0<t \le t_0.
\end{align}
The proof of Theorem \ref{main_theorem} is complete.

%BIBTEX%%%%%%%%
\bibliographystyle{plain}
\bibliography{references}

\begin{thebibliography}{10}

\bibitem{quadratic_evol}
A.~Aleman and J.~Viola.
\newblock On weak and strong solution operators for evolution equations coming from quadratic operators.
\newblock {\em J. Spectr. Theory}, 8(1):33--121, 2018.

\bibitem{PartialGS}
{P.} Alphonse.
\newblock Quadratic {D}ifferential {E}quations: Partial {G}elfand–{S}hilov {S}moothing {E}ffect and {N}ull-{C}ontrollability.
\newblock {\em Journal of the Institute of Mathematics of Jussieu}, page 1–53, 2020.

\bibitem{alphonse2020polar}
{P.} {Alphonse} and {J. Bernier}.
\newblock {Polar {D}ecomposition of {S}emigroups {G}enerated by {N}on-{S}elfadjoint {Q}uadratic {D}ifferential {O}perators and {R}egularizing {E}ffects}.
\newblock page arXiv:1909.03662, September 2019.

\bibitem{PTSymmetric}
{E}. Caliceti, {S}. {G}raffi, {M}. {H}itrik, and {J}. {S}j{\"o}strand.
\newblock Quadratic $\mathcal{PT}$-{S}ymmetric {O}perators with {R}eal {S}pectrum and {S}imilarity to {S}elf-{A}djoint {O}perators.
\newblock {\em Journal of Physics A: Mathematical and Theoretical}, 45(44):444007, Oct 2012.

\bibitem{ExponentialSingularities}
{E}. Carypis and {P}. {W}ahlberg.
\newblock Propagation of {E}xponential {P}hase {S}pace {S}ingularities for {S}chr{\"o}dinger {E}quations with {Q}uadratic {H}amiltonians.
\newblock {\em Journal of Fourier Analysis and Applications}, 23(3):530--571, 2017.

\bibitem{ComplexFIOs}
{L}. {C}oburn, {M}. {H}itrik, and {J}. {S}j{\"o}strand.
\newblock Positivity, {C}omplex {F}{I}{O}s, and {T}oeplitz {O}perators.
\newblock {\em Pure Appl. Anal.}, 1(3):327--357, 2019.

\bibitem{SimonDavies84}
E.~B. Davies and B.~Simon.
\newblock Ultracontractivity and the heat kernel for {S}chr\"{o}dinger operators and {D}irichlet {L}aplacians.
\newblock {\em J. Funct. Anal.}, 59(2):335--395, 1984.

\bibitem{SimonDavies91}
E.~B. Davies and B.~Simon.
\newblock {$L^p$} norms of noncritical {S}chr\"{o}dinger semigroups.
\newblock {\em J. Funct. Anal.}, 102(1):95--115, 1991.

\bibitem{QuadraticOperators}
{M}. Hitrik and {K}. {P}ravda{-}{S}tarov.
\newblock Spectra and {S}emigroup {S}moothing for {N}on-{E}lliptic {Q}uadratic {O}perators.
\newblock {\em Mathematische Annalen}, 344(4):801--846, Jan 2009.

\bibitem{SemiclassicalHypoelliptic}
M.~Hitrik and K.~Pravda-Starov.
\newblock Semiclassical {H}ypoelliptic {E}stimates for {N}on-{S}elfadjoint {O}perators with {D}ouble {C}haracteristics.
\newblock {\em Comm. Partial Differential Equations}, 35(6):988--1028, 2010.

\bibitem{EigenvaluesAndSubelliptic}
M.~Hitrik and K.~Pravda-Starov.
\newblock Eigenvalues and {S}ubelliptic {E}stimates for {N}on-{S}elfadjoint {S}emiclassical {O}perators with {D}ouble {C}haracteristics.
\newblock {\em Ann. Inst. Fourier (Grenoble)}, 63(3):985--1032, 2013.

\bibitem{SubellipticEstimates}
{M}. Hitrik, {K}. {P}ravda{-}{S}tarov, and {J}. {V}iola.
\newblock From {S}emigroups to {S}ubelliptic {E}stimates for {Q}uadratic {O}perators.
\newblock {\em Transactions of the American Mathematical Society}, 370(10):7391--7415, May 2018.

\bibitem{Minicourse}
{M}. Hitrik and {J}. {S}j{\"o}strand.
\newblock Two {M}inicourses on {A}nalytic {M}icrolocal {A}nalysis.
\newblock {\em Springer Proceedings in Mathematics \& Statistics}, pages 483--540, 2018.

\bibitem{resolvent_esimates_elliptic}
M.~Hitrik, J.~Sj\"{o}strand, and J.~Viola.
\newblock Resolvent estimates for elliptic quadratic differential operators.
\newblock {\em Anal. PDE}, 6(1):181--196, 2013.

\bibitem{GeneralizedMehler}
{L}. H{\"o}rmander.
\newblock Symplectic {C}lassification of {Q}uadratic {F}orms, and {G}eneral {M}ehler {F}ormulas.
\newblock {\em Mathematische Zeitschrift}, 219(1):413--449, May 1995.

\bibitem{HormanderIII}
L.~H\"{o}rmander.
\newblock {\em The analysis of linear partial differential operators. {III}}.
\newblock Classics in Mathematics. Springer, Berlin, 2007.
\newblock Pseudo-differential operators, Reprint of the 1994 edition.

\bibitem{KochTataru}
H.~Koch and D.~Tataru.
\newblock {$L^p$} eigenfunction bounds for the {H}ermite operator.
\newblock {\em Duke Math. J.}, 128(2):369--392, 2005.

\bibitem{return2equilibrium}
M.~Ottobre, G.~A. Pavliotis, and K.~Pravda-Starov.
\newblock Exponential return to equilibrium for hypoelliptic quadratic systems.
\newblock {\em J. Funct. Anal.}, 262(9):4000--4039, 2012.

\bibitem{ContractionSemigroup}
K.~Pravda-Starov.
\newblock Contraction {S}emigroups of {E}lliptic {Q}uadratic {D}ifferential {O}perators.
\newblock {\em Math. Z.}, 259(2):363--391, 2008.

\bibitem{SubellipticEstimatesQuadraticDifferentialOperators}
K.~Pravda-Starov.
\newblock Subelliptic {E}stimates for {Q}uadratic {D}ifferential {O}perators.
\newblock {\em Amer. J. Math.}, 133(1):39--89, 2011.

\bibitem{time_dependent}
K.~Pravda-Starov.
\newblock Generalized {M}ehler {F}ormula for {T}ime-{D}ependent {N}on-{S}elfadjoint {Q}uadratic {O}perators and {P}ropagation of {S}ingularities.
\newblock {\em Mathematische Annalen}, 372(3):1335--1382, 2018.

\bibitem{GaborSingularities}
K.~Pravda-Starov, L.~Rodino, and P.~Wahlberg.
\newblock Propagation of {G}abor {S}ingularities for {S}chr{\"o}dinger {E}quations with {Q}uadratic {H}amiltonians.
\newblock {\em Mathematische Nachrichten}, 291(1):128--159, 2018.

\bibitem{Simon80}
B.~Simon.
\newblock Brownian motion, {$L^{p}$} properties of {S}chr\"{o}dinger operators and the localization of binding.
\newblock {\em J. Functional Analysis}, 35(2):215--229, 1980.

\bibitem{Simon81}
B.~Simon.
\newblock Large time behavior of the {$L^{p}$} norm of {S}chr\"{o}dinger semigroups.
\newblock {\em J. Functional Analysis}, 40(1):66--83, 1981.

\bibitem{AnalyticMicrolocal_Analysis}
J.~Sj\"{o}strand.
\newblock Singularit\'{e}s {A}nalytiques {M}icrolocales.
\newblock In {\em Ast\'{e}risque, 95}, volume~95 of {\em Ast\'{e}risque}, pages 1--166. Soc. Math. France, Paris, 1982.

\bibitem{semigroup2resolvent}
J.~Sj\"{o}strand.
\newblock Resolvent estimates for non-selfadjoint operators via semigroups.
\newblock In {\em Around the research of {V}ladimir {M}az'ya. {III}}, volume~13 of {\em Int. Math. Ser. (N. Y.)}, pages 359--384. Springer, New York, 2010.

\bibitem{NonEllipticQuadraticFormsandSemiclassicalEstimates}
J.~Viola.
\newblock Non-{E}lliptic {Q}uadratic {F}orms and {S}emiclassical {E}stimates for {N}on-{S}elfadjoint {O}perators.
\newblock {\em Int. Math. Res. Not. IMRN}, (20):4615--4671, 2013.

\bibitem{SpectralProjectionsAndResolventBounds}
J.~Viola.
\newblock Spectral {P}rojections and {R}esolvent {B}ounds for {P}artially {E}lliptic {Q}uadratic {D}ifferential {O}perators.
\newblock {\em J. Pseudo-Differ. Oper. Appl.}, 4(2):145--221, 2013.

\bibitem{PolynomialSingularities}
P.~Wahlberg.
\newblock Propagation of {P}olynomial {P}hase {S}pace {S}ingularities for {S}chr{\"o}dinger {E}quations with {Q}uadratic {H}amiltonians.
\newblock {\em Mathematica Scandinavica}, 122(1):107--140, Feb. 2018.

\bibitem{Global_Analytic}
F.~{White}.
\newblock {Propagation of Global Analytic Singularities for Schr{\"o}dinger Equations with Quadratic Hamiltonians}.
\newblock {\em arXiv e-prints}, page arXiv:2102.01474, 2021.

\bibitem{SemiclassicalAnalysis}
M.~Zworski.
\newblock {\em Semiclassical {A}nalysis}, volume 138 of {\em Graduate Studies in Mathematics}.
\newblock American Mathematical Society, Providence, RI, 2012.

\end{thebibliography}
%%%%%%%%%%%%%

\end{document}